\documentclass{amsart}

\usepackage{tikz}
\usepackage{amsmath}
\usepackage{amsthm}
\usepackage{amssymb}
\usepackage{mathtools}
\usepackage{enumerate}

\newcommand{\Q}{\mathbb{Q}}
\newcommand{\C}{\mathcal{C}}

\newcommand{\R}{\mathbb{R}}
\newcommand{\N}{\mathbb{N}}
\newcommand{\U}{\mathcal{U}}

\newcommand{\D}{\mathcal{D}}

\newcommand{\dom}{\operatorname{dom}}
\newcommand{\ran}{\operatorname{ran}}
\newcommand{\diam}{\operatorname{diam}}

\newcommand{\supp}{\operatorname{supp}}

\newcommand{\B}{\mathcal{B}}

\newcommand{\A}{\mathcal{A}}

\newcommand{\M}{\mathcal{M}}

\newcommand{\one}{\mathbf{1}}

\newcommand{\phee}{\varphi}
\newcommand{\ol}{\overline}

\newtheorem{theorem}{Theorem}[section]
\newtheorem{lemma}[theorem]{Lemma}

\newtheorem{claim}[theorem]{Claim}
\newtheorem{corollary}[theorem]{Corollary}

\newtheorem{proposition}[theorem]{Proposition}

\theoremstyle{definition}
\newtheorem{definition}[theorem]{Definition}
\newtheorem{question}[theorem]{Question}

\newtheorem{example}[theorem]{Example}

\numberwithin{equation}{section}

\title[Effective Weak Convergence and Tightness]{Effective Weak Convergence and Tightness of Measures in Computable Polish Spaces}
\author{Diego A. Rojas}
\address{Department of Mathematics, University of Dallas, Irving, Texas 75062, United States}
\email{darojas@udallas.edu}
\address{Department of Mathematics and Statistics, Sam Houston State University, Huntsville, Texas 77340, United States}
\email{dar117@shsu.edu}
\thanks{ORCID: 0000-0001-9878-6264}

\begin{document}

\begin{abstract}
Prokhorov's Theorem in probability theory states that a family $\Gamma$ of probability measures on a Polish space is tight if and only if every sequence in $\Gamma$ has a weakly convergent subsequence. 
Due to the highly non-constructive nature of (relative) sequential compactness, however, the effective content of this theorem has not been studied. 
To this end, we generalize the effective notions of weak convergence of measures on the real line due to McNicholl and Rojas to computable Polish spaces. 
Then, we introduce an effective notion of tightness for families of measures on computable Polish spaces. 
Finally, we prove an effective version of Prokhorov's Theorem for computable sequences of probability measures.
\end{abstract}

\maketitle

\section{Introduction}
Recall that a sequence of finite Borel measures $\{\mu_n\}_{n\in\N}$ on a separable metric space $X$ \emph{weakly converges} to a measure $\mu$ if, for every bounded continuous real-valued function $f$ on $X$, $\lim_n\int_Xfd\mu_n=\int_Xfd\mu$. The space $\M(X)$ of finite Borel measures on a separable metric space $X$ forms a separable metric space when equipped with the Prokhorov metric. The Prokhorov metric was introduced by Prokhorov \cite{P56} in 1956 and metrizes the topology of weak convergence. There are multiple equivalent definitions of weak convergence, all summarized in a 1941 theorem due to Alexandroff \cite{A41} commonly known as the Portmanteau Theorem. 

Recently, McNicholl and Rojas \cite{MR23,R24} developed a framework to study the effective theory of weak convergence in $\M(\R)$. The framework consists of several effective definitions of weak convergence in $\M(\R)$ all shown to be equivalent to one another. 
In this paper, we generalize this framework to computable Polish spaces; that is, complete separable metric spaces with a computability structure. The key to generalizing the framework was to develop an effective notion of tightness for families of measures.

Recall that a family of measures $\Gamma\subseteq\M(X)$ is \emph{tight} if for every $\epsilon>0$, there is a compact subset $K_{\epsilon}$ of $X$ such that $\mu(X\setminus K_{\epsilon})<\epsilon$ for all $\mu\in\Gamma$. We seek to study the following result due to Prokhorov.

\begin{theorem}[Prokhorov's Theorem, \cite{P56}]\label{thm:prokh}
    Let $X$ be a Polish space, and let $\Gamma$ be a family of probability measures in $\M(X)$. The following are equivalent.
    \begin{enumerate}
        \item $\Gamma$ is tight.
        \item Every sequence in $\Gamma$ has a weakly convergent subsequence.
    \end{enumerate}
    In particular, a sequence $\{\mu_n\}_{n\in\N}$ of probability measures in $\M(X)$ is tight if and only if $\{\mu_n\}_{n\in\N}$ has a weakly convergent subsequence.
\end{theorem}

Prokhorov's Theorem requires the completeness assumption on $X$ in order to find a necessary and sufficient condition for a family of probability measures to be tight. The issue is that the proofs of Prokhorov's Theorem available in the literature are highly ineffective. This is a consequence of dealing with (relative) sequential compactness. In general, theorems concerning sequential compactness such as the Bolzano-Weierstrass Theorem \cite{BGM12}, the Arzel\`a-Ascoli Theorem \cite{A14a}, and Helly's Selection Theorem \cite{A14b}, are highly ineffective. Thus, the main question we seek to answer is the following:

\begin{question}
Is there an effective version of Theorem \ref{thm:prokh}?
\end{question}

In this paper, we focus on Theorem \ref{thm:prokh} for sequences of probability measures. We show that it is possible to compute an effectively weakly convergent subsequence of an effectively tight sequence with additional computable information obtained in a non-uniform manner. However, we also exhibit a counterexample to the converse: a computable sequence of probability measures containing an effectively weakly convergent subsequence that is not effectively tight.

The paper is structured as follows. Section 2 contains the necessary background in computable analysis and computable measure theory. We then introduce effective tightness of measures in Section 3. In Section 4, we generalize effective weak convergence of measures to computable Polish spaces using effective tightness. In Section 5, we provide in full detail the proof of an effective version of one direction of Theorem \ref*{thm:prokh} as well as a counterexample to the converse. Finally, we offer concluding remarks and open questions in Section 6.

\section{Background}

Let $(X,d)$ be a separable metric space. In this paper, we denote the set of all continuous real-valued functions on $X$ by $C(X)$, the set of all bounded continuous real-valued functions on $X$ by $C_b(X)$, and the set of all compactly-supported continuous real-valued functions on $X$ by $C_K(X)$. 
We define the \emph{support} of a function $f\in C(X)$ to be the set $\supp{f}=\overline{X\setminus f^{-1}(\{0\})}$, where $\overline{A}$ denotes the closure of $A$. For $A\subseteq X$, denote by $\one_A$ the function given by
\[
    \one_A(x)=\begin{cases}
        1 & \text{if }x\in A,\\
        0 & \text{if }x\not\in A.
    \end{cases}
\]

For $x\in X$ and $A\subseteq X$, let $d(x,A)=\inf_{a\in A}d(x,a)$. 
For $\epsilon>0$, we denote the open ball of radius $\epsilon$ centered at $a\in X$ by $B(a,\epsilon)$ and  the closed ball of radius $\epsilon$ centered at $a$ by $\overline{B}(a,\epsilon)$.
For $\epsilon>0$ and $A\subseteq X$, $B(A,\epsilon)=\bigcup_{a\in A}B(a,\epsilon)$ is called the open $\epsilon$-neighborhood of $A$.
We denote the Borel $\sigma$-algebra of $X$ by $\B(X)$. 
The Prokhorov metric $d_P$ is defined as follows: for any $\mu,\nu\in\M(X)$, $d_P(\mu,\nu)$ is the infimum over all $\epsilon>0$ such that 
\[\mu(A)\leq\nu(B(A,\epsilon))+\epsilon \ \text{and} \ \nu(A)\leq\mu(B(A,\epsilon))+\epsilon
\] 
for all $A\in\B(X)$. For $x\in X$, denote by $\delta_x$ the Dirac measure on $x$: for any $A\in\B(X)$,
\[
    \delta_x(A)=\begin{cases}
        1 & \text{if }x\in A,\\
        0 & \text{if }x\not\in A.
    \end{cases}
\]

Fix a measure $\mu\in\M(X)$, and let $A \subseteq X$. $A$ is said to be a \emph{$\mu$-continuity set} if $\mu(\partial A)=0$. 

Below, we state a version of the classical Portmanteau Theorem. 
Although there are as many as ten equivalent definitions of weak convergence (see \cite{LKN16}), we will focus on five.

\begin{theorem}[Classical Portmanteau Theorem, \cite{A41}]\label{thm:cpt}
Let $\{\mu_n\}_{n\in\N}$ be a sequence in $\M(X)$. The following are equivalent.
	\begin{enumerate}
    		\item $\{\mu_n\}_{n\in\N}$ weakly converges to $\mu$.
    		\item For every uniformly continuous $f\in C_b(X)$,
    		$\lim\limits_{n\rightarrow\infty}\int_{X}fd\mu_n=\int_{X}fd\mu$.
    		\item For every closed $C\subseteq X$, 
			$\limsup\limits_{n\rightarrow\infty}\mu_n(C)\leq\mu(C)$.
    		\item For every open $U\subseteq X$, 
			$\liminf\limits_{n\rightarrow\infty}\mu_n(U)\geq\mu(U)$.
    		\item For every $\mu$-continuity set $A\subseteq X$, 
			$\lim\limits_{n\rightarrow\infty}\mu_n(A)=\mu(A)$.
	\end{enumerate}
\end{theorem}

\subsection{Background from computable analysis}

We presume familiarity with the fundamentals of computability theory as expounded in \cite{Cooper.2004}.  
For a more expansive treatment of computable analysis, see \cite{BH21,Weihrauch.2000}.

A \emph{computable metric space} is a triple $(X,d,\alpha)$ with the following properties:
\begin{enumerate}
    \item $(X,d)$ is a separable metric space;
    \item $\alpha:\subseteq\N\rightarrow X$ is an effective enumeration such that $\dom{\alpha}$ is computable and $\overline{\ran{\alpha}}=X$;
    \item $d(\alpha(i),\alpha(j))$ is computable uniformly in $i$ and $j$.
\end{enumerate}
The points in $\ran{\alpha}$ are known as \emph{rational points}. 
If $(X,d)$ is complete, we call $(X,d,\alpha)$ a \emph{computable Polish space}. 

\begin{example}
    For each $n\in\N$, Euclidean space $(\R^n,|\cdot|,\alpha_{\Q^n})$, where $\alpha_{\Q^n}:\N\rightarrow\Q^n$ is an enumeration of all rational vectors in $\R^n$, is a computable Polish space.
\end{example}

\begin{example}
    The Cantor space $(2^{\omega},d_H,\alpha_{2^{<\omega}})$, where $\alpha_{2^{<\omega}}:\subseteq\N\rightarrow2^{\omega}$ is an enumeration of all sequences of the form $\sigma0^{\omega}$ with $\sigma\in2^{<\omega}$ and $d_H$ is the Hamming distance on $2^{\omega}$, is a computable Polish space.
\end{example}

\begin{example}
    If $(X,d,\alpha)$ is a computable Polish space, then so is the space $(\M(X),
    d_P,\alpha_{\D})$, where $\alpha_{\D}:\subseteq\N\rightarrow\M(X)$ is an enumeration of all finite rational linear combinations of Dirac measures on $\ran\alpha$.
\end{example}

Throughout the rest of this paper, we let $(X,d,\alpha)$ be a computable Polish space with $\alpha(i)=s_i$ for each $i\in\N$. 

Let $\N^{\N}$ denote the space of infinite sequences of natural numbers. 
Note that any naming system can be encoded as sequences in $\N^{\N}$ \cite{BH21}. 
Let $\Phi_e$ denote the $e$th partial computable functional on $\N^{\N}$, and let $\Phi_e^z$ denote the $e$th partial computable functional on $\N^{\N}$ relative to an oracle $z\in\N^{\N}$. 
Fix $x\in X$.  A \emph{(Cauchy) name} of $x$ is a sequence $\{s_i\}_{i \in \N}$ of rational points in $X$ so that 
$\lim_i s_i = x$ and so that $d(s_i, s_{i+1}) < 2^{-i}$ for all $n \in \N$.  
$x$ is \emph{computable} if it has a computable name.  
An index of such a name is also said to be an index of $x$.

A sequence $\{x_n\}_{n \in \N}$ in $X$ is \emph{computable} if $x_n$ is computable uniformly in $n$.
If $\{x_n\}_{n\in\N}$ converges to a point $x\in X$ in $d$, a \emph{modulus of convergence} for $\{x_n\}_{n\in\N}$ is a function $\epsilon:\N\rightarrow\N$ such that for all $N\in\N$ and all $n\geq N$, $n\geq\epsilon(N)$ implies $d(x_n,x)<2^{-N}$.

Let $x$ be a real number. Say
$x$ is \emph{left-c.e.} (\emph{right-c.e.}) if its left (right) Dedekind cut is c.e.\  
It follows that $x$ is computable if and only if it is left-c.e. and right-c.e.

The following example of an incomputable left-c.e. real number is due to Specker \cite{S49}. 
Let $A$ be an incomputable c.e. set, and let $\{a_n\}_{n\in\N}$ be an effective enumeration of $A$. 
Then, $\{\sum_{i=0}^n2^{-(a_i+1)}\}_{n\in\N}$ is a computable increasing sequence of rationals called a \emph{Specker sequence}. 
Let $x=\lim_n\sum_{i=0}^n2^{-(a_i+1)}=\sum_{i=0}^{\infty}2^{-(a_i+1)}$. 
Since $x$ is the limit of a computable increasing sequence of rationals, it follows that $x$ is left-c.e. \ 
Moreover, $x$ is incomputable since $A$ is an incomputable set.

Suppose $\{a_n\}_{n \in \N}$ is a sequence of reals, and let $g : \subseteq \Q \rightarrow \N$. 
We say $g$ \emph{witnesses that $\liminf_n a_n$ is not smaller than $a$} if $\dom(g)$ is the left Dedekind cut of $a$ and if $r < a_n$ whenever $r \in \dom(g)$ and $n \geq g(r)$. 
We say $g$ \emph{witnesses that $\limsup_n a_n$ is not larger than $a$} if $\dom(g)$ is the right Dedekind cut of $a$ and if $r > a_n$ whenever $r \in \dom(g)$ and $n \geq g(r)$.
The following observation can be found in \cite{MR23}.

\begin{proposition}\label{prop:wit.2.mod}
    Suppose there is a computable witness that $\liminf_n a_n$ is not smaller than $a$, and suppose there is a computable
    witness that $\limsup_n a_n$ is not larger than $a$.  Then, $\lim_n a_n = a$, and $\{a_n\}_{n \in \N}$ has a computable modulus of convergence.
\end{proposition}

A \emph{(functional-oracle) name} of a function $f\subseteq X\rightarrow\R$ is a pair $(e,z)\in\N\times\N^{\N}$ such that, if $\rho\in\N^{\N}$ is a name of $x\in X$, $\Phi_e^z(\rho)$ is a name of $f(x)\in\R$. 
We say $f$ is \emph{computable} if it has a computable name.
An index of $f$ is defined to be an index of a partial computable functional associated with $f$.
We denote the set of all bounded computable real-valued functions on $X$ by $C_b^c(X)$ and the set  of all compactly-supported computable real-valued functions on $X$ by $C_K^c(X)$.  A function $f: \subseteq X \rightarrow \R$ is \emph{lower semicomputable} if there is a partial computable functional $\Phi$ so that 
$\Phi(\rho)$ enumerates the left Dedekind cut of $f(x)$ whenever $\rho$ is a Cauchy name of $x$.
A function $f: \subseteq X \rightarrow \R$ is \emph{upper semicomputable} if $-f$ is lower semicomputable.

A function $F:\subseteq C(X)\rightarrow\R$ is \emph{computable} if there is a partial computable functional $\Phi$ 
so that $\Phi(\rho)$ is a Cauchy name of $F(f)$ whenever $\rho$ is a name of $f$.  
An index of such a functional $\Phi$ is also said to be an index of $F$.


Fix an effective enumeration $\{I_i\}_{i\in\N}$ of all rational open balls in $X$.
An open set $U\subseteq X$ is $\Sigma^0_1$ if $U=\bigcup_{i\in E}I_i$ for some c.e. set $E\subseteq\N$.
We denote the set of $\Sigma^0_1$ subsets of $X$ by $\Sigma^0_1(X)$.  
We say that $e \in \N$ \emph{indexes} $U \in \Sigma_1^0(X)$ if $U=\bigcup_{i\in W_e}I_i$, where $W_e$ denotes the $e$th c.e. set. 

A closed set $C\subseteq X$ is $\Pi^0_1$ if $X\setminus C$ is $\Sigma^0_1$. 
We denote the set of $\Pi^0_1$ subsets of $X$ by $\Pi^0_1(X)$.
Indices of sets in $\Pi^0_1(X)$ are defined analogously. 

Given an effective enumeration $\{I_i\}_{i\in\N}$ of all rational open balls in $X$ and an effective enumeration $\{\Delta_j\}_{j\in\N}$ of all finite subsets of $\N$, define the $j$th \emph{rational open set} $J_j$ by $J_j=\bigcup_{i\in\Delta_j}I_i$.
Thus, it is possible to effectively enumerate all rational open sets in $X$. 
For a compact set $K\subseteq X$, a \emph{(minimal cover) name} for $K$ is an enumeration of rational open sets containing $K$. We say $K$ is \emph{computably compact} if it has a computable name. An index of $K$ is defined to be an index of a name of $K$.

A $G_{\delta}$ set $G\subseteq X$ is $\Pi^0_2$ if there is a uniformly $\Sigma^0_1$ sequence $\{U_n\}_{n\in\N}$ such that $G=\bigcap_nU_n$. We will make use of the following theorem, which can be found in \cite{HR09a}.

\begin{theorem}[Computable Baire Category Theorem]\label{thm:cbct}
    Every dense $\Pi^0_2$ subset of a computable metric space contains a computable dense sequence.
\end{theorem}

\subsection{Background from computable measure theory}

A measure $\mu\in\M(X)$ is \emph{computable} if $\mu(X)$ is a computable real and $\mu(U)$ is left-c.e. uniformly in an index of $U\in\Sigma^0_1(X)$;  i.e. it is possible to compute an index of the left Dedekind cut of $\mu(U)$ from an index of $U$.  A sequence of measures $\{\mu_n\}_{n\in\N}$ in $\M(X)$ is \emph{uniformly computable} if $\mu_n$ is a computable measure uniformly in $n$.

Corollary 4.3.1 in \cite{HR09b} gives us a way to characterize computable measures in terms of their integrals. The following proposition is a useful extension of this characterization. 

\begin{proposition}\label{prop:int}
For $\mu\in\M(X)$, the following are equivalent.
\begin{enumerate}
    \item $\mu$ is computable.
    \item $f\mapsto\int_Xf \ d\mu$ is computable on nonnegative $f\in C_b^c(X)$. That is, from an index of $f\in C_b^c(X)$ and a bound $B\in\N$ on $f$, it is possible to compute an index of $\int_Xf \ d\mu$.
    \item $f\mapsto\int_X f\ d\mu$ is computable on uniformly continuous $f\in C_b^c(X)$ such that $0\leq f\leq1$.
\end{enumerate}
\end{proposition}

The proof of Proposition \ref{prop:int} can be found in \cite{MR23} for the case $X=\R$ but may be generalized to any computable Polish space $X$.

Suppose $\mu \in \M(X)$ is computable. A pair $(U,V)$ of $\Sigma^0_1$ subsets of $X$ is \emph{$\mu$-almost decidable} if $U\cap V=\emptyset$, $\mu(U\cup V)=\mu(X)$, and $\ol{U\cup V}=X$.  
Given $A\in\B(X)$, we say that $(U,V)$ is a \emph{$\mu$-almost decidable pair of $A$} if, in addition, $U \subseteq A \subseteq X \setminus V$.  
We then say $A$ is \emph{$\mu$-almost decidable} if it has a $\mu$-almost decidable pair. 
Suppose $(U,V)$ is a $\mu$-almost decidable pair of $A$. Then, $e$ indexes $A$ if $e=\left<i,j\right>$ for some index $i$ of $U$ and some index $j$ of $V$.  We note that $\mu$-almost decidable sets are effective analogues of $\mu$-continuity sets.  The definition of $\mu$-almost decidable set is from \cite{R18}.

In this paper, we will use the following theorem from \cite{MR23} on approximating $\Pi^0_1$ subsets of $X$ with almost decidable sets. The result was originally proven for the case $X=\R$, but we obtain the result \emph{mutatis mutandis}.
\begin{theorem}\label{thm:apprx}
    Suppose $\mu \in \mathcal{M}(X)$ is computable, and let $C \in \Pi^0_1(X)$.
    Then, for every $r \in (\mu(C), \infty)\cap\Q$ there is a $\mu$-almost decidable set $B \supseteq C$ so 
    that $\mu(B) < r$. 
    Furthermore, an index of $B$ can be computed from $r$ and an index of $C$.
\end{theorem}

\section{Effective Tightness of Measures}

Classically, a family of measures $\Gamma\subseteq\M(X)$ is \emph{tight} if, for every $\epsilon>0$, there exists a compact $K\subseteq X$ such that $\mu(X\setminus K)<\epsilon$ for every $\mu\in\Gamma$. An effective definition of tightness would require a computable way to witness the tightness condition on a family of measures. To this end, we propose the following definition.

\begin{definition}
    A family of computable measures $\Gamma\subseteq\M(X)$ is \emph{effectively tight} if there is an effective procedure that, on input $N\in\N$, computes an index of a computably compact subset $K$ of $X$ such that $\mu(X\setminus K)<2^{-N}$ for all $\mu\in\Gamma$.
\end{definition}

Under this definition, every effectively tight sequence of measures is tight. In practice, to show that $\Gamma$ is effectively tight, it is enough to have an effective procedure that, on input $N\in\N$, computes an index of a finite union $B$ of rational closed balls of $X$ such that $\mu(X\setminus B)<2^{-N}$ for all $\mu\in\Gamma$. An \emph{effective tightness witness} for $\Gamma$ is a computable sequence $\{K_N\}_{N\in\N}$ of computably compact sets (or finite unions of rational closed balls) such that $\mu(X\setminus K_N)<2^{-N}$ for all $\mu\in\Gamma$. Thus, a family $\Gamma$ is effectively tight if it admits an effective tightness witness.

Below, we provide several examples of effectively tight families of measures. We start with the following observation.

\begin{proposition}\label{prop:finite}
If $\mu\in\M(X)$ is computable, then $\{\mu\}$ is effectively tight.
\end{proposition}


This result was originally proven for computable probability measures on $X$ by Hoyrup and Rojas (Theorem 2, \cite{HR09c}) and can be easily extended to any computable measure in $\M(X)$. By induction on $n$ and the fact that computably compact sets are closed under finite union, we get the following immediate corollary.

\begin{corollary}\label{cor:finite}
    If $\Gamma\subseteq\M(X)$ is finite and each $\mu\in\Gamma$ is computable, then $\Gamma$ is effectively tight.
\end{corollary}


As shown below, a necessary and sufficient condition for effective tightness of a computable sequence of measures in $\M(X)$ is to effectively concentrate all but finitely many measures in the sequence.

\begin{proposition}\label{prop:tight}
    Let $\{\mu_n\}_{n\in\N}$ be a computable sequence of measures. The following are equivalent.
    \begin{enumerate}
        \item $\{\mu_n\}_{n\in\N}$ is effectively tight.
        \item There is an effective procedure that, on input $N\in\N$, computes an index of a computably compact subset $K$ of $X$ and an $n_0\in\N$ such that $\mu_n(X\setminus K)<2^{-N}$ for all $n\geq n_0$. 
    \end{enumerate}
\end{proposition}

\begin{proof}
    It suffices to prove the converse. Suppose there is an effective procedure that, on input $N\in\N$, computes an index of a computably compact subset $K$ of $X$ and an $n_0$ such that $\mu_n(X\setminus K)<2^{-N}$ for all $n\geq n_0$. By Corollary \ref*{cor:finite}, it is possible to compute, on input $N\in\N$, an index of a computably compact subset $C$ of $X$ such that $\mu_k(X\setminus C)<2^{-N}$ for all $0\leq k<n_0$. On input $N\in\N$, compute an index of $K\cup C$. Since $K$ and $C$ are computably compact, so is $K\cup C$. Furthermore, $\mu_n(X\setminus(K\cup C))<2^{-N}$ for all $n$.
\end{proof}

An \emph{effective tail-tightness witness} for a computable sequence $\{\mu_n\}_{n\in\N}$ in $\M(X)$ is a computable sequence $\{(K_N,n_N)\}_{N\in\N}$ of pairs consisting of acomputably compact set (or finite union of rational closed balls) $K_N$ together with an index $n_N$ such that $\mu_n(X\setminus K_N)<2^{-N}$ for all $n\geq n_N$. Thus, Proposition \ref{prop:tight} states that a computable sequence $\{\mu_n\}_{n\in\N}$ in $\M(X)$ admits an effective tightness witness if and only if it admits an effective tail-tightness witness.

\section{Effective Weak Convergence of Measures}

Having established an effective notion of tightness for families of measures in $\M(X)$, we now proceed to establish an effective framework for weak convergence of measures on computable Polish spaces. As in \cite{MR23}, we provide both a non-uniform and a uniform version.

\begin{definition}
    Let $\{\mu_n\}_{n\in\N}$ be a sequence of measures in $\M(X)$. We say $\{\mu_n\}_{n\in\N}$ \emph{effectively weakly converges to} $\mu$ if for every $f\in C^c_b(X)$, $\lim_n\int_Xfd\mu_n=\int_Xfd\mu$ and it is possible to compute an index of a modulus of convergence of $\{\int_Xfd\mu_n\}_{n\in\N}$ from an index of $f$ and a bound $B\in\N$ on $|f|$.
\end{definition}

\begin{definition}
    Let $\{\mu_n\}_{n\in\N}$ be a sequence of measures in $\M(X)$. We say $\{\mu_n\}_{n\in\N}$ \emph{uniformly effectively weakly converges to} $\mu$ if $\{\mu_n\}$ weakly converges to $\mu$, and it is possible to compute a modulus of convergence of $\{\int_Xfd\mu_n\}_{n\in\N}$ from a name of $f\in C_b(X)$ and a bound $B\in\N$ on $|f|$.
\end{definition}

\subsection{Equivalence of Effective Weak Convergence Notions}

The main goal is to show that, under our definition of effective tightness, the two definitions are equivalent.

\begin{theorem}\label{thm:gen.ewc}
    Let $\{\mu_n\}_{n\in\N}$ be a computable sequence in $\M(X)$. The following are equivalent.
    \begin{enumerate}
        \item $\{\mu_n\}_{n\in\N}$ effectively weakly converges to $\mu$.
        \item $\{\mu_n\}_{n\in\N}$ uniformly effectively weakly converges to $\mu$.
    \end{enumerate}
\end{theorem}

First, we need the following proposition.

\begin{proposition}\label{prop:ewc2et}
    Every effectively weakly convergent computable sequence of measures is effectively tight.
\end{proposition}

\begin{proof}
    Let $\{\mu_n\}_{n\in\N}$ be a computable sequence in $\M(X)$ that effectively weakly converges to $\mu$. By assumption, $\mu$ is computable. Thus, by Proposition \ref*{prop:finite}, it is possible to compute an increasing sequence $\{K_m\}_{m\in\N}$ of computably compact subsets of $X$ such that $\mu(X\setminus K_m)<2^{-m}$ uniformly in $m$. From $\{K_m\}_{m\in\N}$, compute a sequence $\{J_m\}_{m\in\N}$ such that $J_m$ is a rational open cover of $K_m$ uniformly in $m$. Let $C_m=X\setminus J_m$ for each $m\in\N$. Then, $\{C_m\}_{m\in\N}$ is a uniformly computable sequence of computably closed subsets of $X$ such that for each $m\in\N$, $C_m\cap K_m=\emptyset$. For each $m$, define the function $f_m:X\rightarrow[0,1]$ given by
    \[
    f_m(x)=\dfrac{d(x,K_m)}{d(x,K_m)+d(x,C_m)}.
    \]
    Then, $\{f_m\}_{m\in\N}$ is a uniformly computable sequence of functions such that for all $m\in\N$, $f_m(x)=0$ if $x\in K_m$, $f_m(x)=1$ if $x\in C_m$, and $0<f_m(x)<1$ otherwise. 
    
    Now, for each $m\in\N$, define the function $g_m:X\rightarrow[0,1]$ by 
    \[
        g_m(x)=\min\{1,\max\{f_m(x),f_{m+1}(x)\}\}.
    \] 
    Then, $\{g_m\}_{m\in\N}$ is a computable sequence such that $\one_{X\setminus K_{m+1}}\leq g_m\leq\one_{X\setminus K_m}$ for all $m\in\N$. It follows that for all $\nu\in\M(X)$ and $m\in\N$,
    \[
    \nu(X\setminus K_{m+1})\leq\int_{X}g_m \ d\mu\leq\nu(X\setminus K_m).
    \]
    Now, fix $N\in\N$. We first search for $m_0\geq N$ so that $\int_Xg_{m_0}\ d\mu<2^{-N}$. Since $\mu$ is finite, it follows that this search must terminate. Since $\mu$ is computable, this search is effective. Set $m_1=m_0+1$. Since $\{\mu_n\}_{n\in\N}$ effectively weakly converges to $\mu$, we can now compute an $n_0\in\N$ so that $\int_Xg_{m_0}\ d\mu_n<2^{-N}$ for all $n\geq n_0$. Thus, $\mu_n(X\setminus K_{m_1})<2^{-N}$ for all $n\geq n_0$. By Proposition \ref*{prop:tight}, it follows that $\{\mu_n\}_{n\in\N}$ is effectively tight. 
\end{proof}

\begin{lemma}\label{lem:caf}
    Let $\{\mu_n\}_{n\in\N}$ be a computable sequence in $\M(X)$ that effectively weakly converges to $\mu$. From an index of $f\in C_b^c(X)$ and a bound $B$ on $|f|$, it is possible to compute an index of a computably compact set $K\subseteq X$, an $n_0\in\N$, and a function $\phee\in C_b(X)$ such that $\supp\phee=K$, $|\int_X(f-\phee)d\mu|<2^{-N}$, and $|\int_X(f-\phee)d\mu_n|<2^{-N}$ for all $n\geq n_0$.
\end{lemma}

\begin{proof}
    By Proposition \ref{prop:finite}, it is possible to compute an index of a computably compact $K_1\subseteq X$ such that $\mu(X\setminus K_1)<B^{-1}2^{-(N+2)}$ on input $N$. By Proposition \ref*{prop:ewc2et}, it is possible to compute an index of a computably compact $K_2\subseteq X$ and an index $n_0$ such that $\mu(X\setminus K_2)<B^{-1}2^{-(N+2)}$ and $\mu_n(X\setminus K)<B^{-1}2^{-(N+2)}$ for all $n\geq n_0$ on input $N$. 
    Let $K=K_1\cup K_2$. 
    Let $J$ be a rational open cover of $K$, and let $C=X\setminus J$. 
    Define the function $g:X\rightarrow[0,1]$ by
    \[
    g(x)=\dfrac{d(x,C)}{d(x,C)+d(x,K)}.
    \]
    Then, $g$ is a bounded computable function since both $C$ and $K$ are computably closed. 
    Since $\{\mu_n\}_{n\in\N}$ effectively weakly converges to $\mu$, it is possible to compute an index $n_1\geq n_0$ such that 
    \[
        \left|\int_Xg\ d\mu_n-\int_Xg\ d\mu\right|<1
    \]
    for all $n\geq n_1$. 

    Fix $f\in C_b(X)$, $B\geq|f|$, and $N\in\N$. Let $\rho$ be a name of $f$ and $\kappa$ a name of $K$. From $\rho$, $B$, and $\kappa$, it is possible to compute a modulus of uniform continuity $\omega$ for $f$ on $K$. Then, from $\rho$, $\kappa$, and $\omega$, it is possible to compute a name $\rho_K$ of $f|_K$. 
    Now, let 
    \[
    \epsilon=\dfrac{2^{-(N+1)}}{1+\int_Xg\ d\mu}.
    \] 
    Let $N_0$ be least such that $d(x,y)<2^{-\omega(N_0)}$ implies $d(f(x),f(y))<\epsilon$ for all $x,y\in K$. 
    Since $K$ is computably compact, we can compute from $\kappa$ a finite list of points $\{x_0,\ldots,x_M\}\subseteq K$ with 
    \[
    K\subseteq\bigcup_{i=0}^MB(x_i,2^{-\omega(0)}).
    \]
    Let $\phee:K\rightarrow\R$ be given by $\phee(x)=\max\{-B,\min\{f(x_i),B\}\}$ whenever $x\in B(x_i,2^{-\omega(0)})$. By definition, $\phee$ is a computable function such that $|\phee|\leq B$. Furthermore, for all $x\in K$, we can find $0\leq i\leq M$ such that
    \[
    |f(x)-\phee(x)|=|f(x)-\max\{-B,\min\{f(x_i),B\}\}|\leq|f(x)-f(x_i)|<\epsilon.
    \]
    Thus, $\max\{|f(x)-\phee(x)|:x\in K\}<\epsilon$

    Let $n\in\N$, and suppose $n\geq n_1$. By construction of $g$, $\mu(K)<1+\int_Xgd\mu$. However, since $n\geq n_1$,
    \begin{align*}
        \mu_n(K)&\leq\int_Xg\ d\mu_n\\
        &\leq\left|\int_Xg\ d\mu_n-\int_Xg\ d\mu\right|+\int_Xg\ d\mu\\
        &<1+\int_Xg\ d\mu.
    \end{align*}
    If $\nu\in\M(X)$, then
    \begin{align*}
        \left|\int_X(f-\phee)d\nu\right|&\leq\int_K|f-\phee|d\nu+\int_{X\setminus K}|f-\phee|d\nu\\
        &\leq\dfrac{2^{-(N+1)}}{1+\int_Xg\ d\mu}\nu(K)+2B\nu(X\setminus K).
    \end{align*}
    Thus,
    \begin{align*}
        \left|\int_X(f-\phee)d\mu\right|&\leq\dfrac{2^{-(N+1)}}{1+\int_Xg\ d\mu}\mu(K)+2B\mu(X\setminus K)\\
        &<\dfrac{2^{-(N+1)}}{1+\int_Xg\ d\mu}\left(1+\int_Xg\ d\mu\right)+2B\left(\dfrac{2^{-(N+1)}}{2B}\right)\\
        &=2^{-N}.
    \end{align*}
    Furthermore, for all $n\geq n_1$,
    \begin{align*}
        \left|\int_X(f-\phee)d\mu_n\right|&\leq\dfrac{2^{-(N+1)}}{1+\int_Xg\ d\mu}\mu_n(K)+2B\mu_n(X\setminus K)\\
        &<\dfrac{2^{-(N+1)}}{1+\int_Xg\ d\mu}\left(1+\int_Xg\ d\mu\right)+2B\left(\dfrac{2^{-(N+1)}}{2B}\right)\\
        &=2^{-N}.
    \end{align*}
\end{proof}

\begin{proof}[Proof of Theorem \ref*{thm:gen.ewc}]
    It is possible to compute a name of $f\in C_b(X)$ from an index of $f$. It follows that every uniformly effectively weakly convergent sequence effectively weakly converges.

    Now, suppose $\{\mu_n\}_{n\in\N}$ effectively weakly converges to $\mu$. Let $B\in\N$, and suppose $\rho$ is a name of $f\in C_b(X)$ with $|f|\leq B$. 
    Fix $N\in\N$. We construct a function $G:\N\rightarrow\N$ as follows. By means of Lemma \ref{lem:caf}, it is possible to compute an index of a computably compact subset $K$ of $X$, an index $n_1$, and an index of a bounded computable function $\phee$ supported on $K$ such that 
    \[
    \left|\int_X(f-\phee)d\mu\right|<2^{-(N+2)} \ \text{and} \ \left|\int_X(f-\phee)d\mu_n\right|<2^{-(N+2)}
    \] 
    for all $n\geq n_1$. Since $\{\mu_n\}_{n\in\N}$ effectively weakly converges to $\mu$, it is possible to compute an index $n_2\geq n_1$ such that $|\int_X\phee\ d\mu_n-\int_X\phee\ d\mu|<2^{-(N+1)}$ for all $n\geq n_2$. Set $G(N)=n_2$.

    Suppose $n\geq G(N)$. Then,
    \begin{align*}
        \left|\int_Xf\ d\mu_n-\int_Xf\ d\mu\right|&\leq\left|\int_X(f-\phee) d\mu_n\right|+\left|\int_X\phee\ d\mu_n-\int_X\phee\ d\mu\right|+\left|\int_X(\phee-f)d\mu\right|\\
        &<2^{-(N+2)}+2^{-(N+1)}+2^{-(N+2)}\\
        &=2^{-N}.
    \end{align*}
    Thus, $G$ is a modulus of convergence for $\{\int_Xf\ d\mu_n\}_{n\in\N}$. Since the construction of $G$ from $\rho$ and $B$ is uniform, it follows that $\{\mu_n\}_{n\in\N}$ uniformly effectively weakly converges to $\mu$.
\end{proof}

In \cite{MR23}, we witnessed for the first time the equivalence between non-uniform and uniform effective notions of weak convergence for the case $X=\R$. Theorem \ref{thm:gen.ewc} shows that this equivalence occurs in all computable Polish spaces. We also obtain the following corollary.

\begin{corollary}\label{cor:ewc2wc}
    If a computable sequence in $\M(X)$ effectively weakly converges, then it weakly converges.
\end{corollary}

Note that the completeness assumption on $X$ is necessary because, otherwise, we could construct measures on $X$ that are not (effectively) tight.

\subsection{Effective Portmanteau Theorem}

We say that $\{\mu_n\}_{n\in\N}$ \emph{(uniformly) effectively weakly converges to} $\mu$ \emph{on} $\C\subseteq C_b(X)$ if we replace $C_b(X)$ with $\C$ in the definition of (uniform) effective weak convergence.

We now proceed to prove an effective version of Theorem \ref{thm:cpt}.

\begin{theorem}[Effective Portmanteau Theorem]\label{thm:ept1}
    Let $\{\mu_n\}_{n\in\N}$ be a uniformly computable sequence in $\M(X)$, and let $\mu \in \M(X)$. The following are equivalent.
        \begin{enumerate}
               \item $\{\mu_n\}_{n\in\N}$ effectively weakly converges to $\mu$.\label{thm:ept1::ewc}
            
                \item $\{\mu_n\}_{n \in \N}$ effectively weakly converges to $\mu$ on uniformly continuous functions in $C_b(X)$.  \label{thm:ept1::uc}
        
                \item $\mu$ is computable, and from an index of $C \in \Pi^0_1(X)$ it is possible to compute an index of a 
            witness that $\limsup_n \mu_n(C)$ is not larger than $\mu(C)$.\label{thm:ept1::clsd}
            
                \item	$\mu$ is computable, and from an index of $U \in \Sigma^0_1(X)$ it is possible to compute an index of 
            a witness that $\liminf_n \mu_n(U)$ is not smaller than $\mu(U)$.\label{thm:ept1::opn}
            
                \item $\mu$ is computable, and for every $\mu$-almost decidable $A$, $\lim_n \mu_n(A) = \mu(A)$ and an index of a modulus of convergence of $\{\mu_n(A)\}_{n \in \N}$ can be computed from a $\mu$-almost decidable index of $A$.\label{thm:ept1::a.d.}
        \end{enumerate}
\end{theorem}

When $f \in C_b(X)$ and $t \in\R$, let $U^f_t = \{f > t\}$, and let $\ol{U}^f_t = \{f \geq t\}$.  By a standard argument with Tonelli's Theorem, if $0 \leq f \leq 1$, then 
\[
\int_X f\ d\nu = \int_0^1 \nu(U^f_t)\ dt = \int_0^1 \nu(\ol{U}^f_t)\ dt
\]
whenever $\nu \in \M(X)$.

To prove Theorem \ref{thm:ept1}, we will need the following lemmas. The proofs of both these lemmas can be found in \cite{MR23}.

\begin{lemma}\label{lm:sc.U}
Let $f\in C^c_b(X)$ satisfy the condition $0<f<1$.  Fix a computable $\nu \in \M(X)$.  
\begin{enumerate}
    \item The function $t \mapsto \nu(U^f_t)$ is lower-semicomputable uniformly in indices of $f$ and $\nu$. \label{lm:sc.U::lsc}
    \item The function $t \mapsto \nu(\ol{U}^f_t)$ is upper-semicomputable uniformly in indices of $f$ and $\nu$.
    \label{lm:sc.U::usc}
\end{enumerate}
\end{lemma}

\begin{lemma}\label{lm:mod.int}
Let $\{\mu_n\}_{n\in\N}$ be a computable sequence in $\M(X)$ that weakly converges to a computable measure $\mu$. Furthermore, let $f\in C_b^c(X)$ satisfy the condition $0<f<1$.
\begin{enumerate}
    \item Suppose that from an index of $U \in \Sigma^0_1(X)$ it is possible to compute an index of a witness that 
    $\liminf_n \mu_n(U)$ is not smaller than $\mu(U)$.  Then, there is a computable witness that 
    $\liminf_n \int_X f\ d\mu_n$ is not smaller than $\int_X f\ d\mu$.
    \label{lm:mod.int::lower}
    
    \item Suppose that from an index of $C \in \Pi^0_1(X)$ it is possible to compute an index of a witness that 
    $\limsup_n \mu_n(C)$ is not larger than $\mu(C)$.  Then, there is a computable witness that 
    $\limsup_n \int_X f\ d\mu_n$ is not larger than $\int_X f\ d\mu$.
    \label{lm:mod.int::upper}
\end{enumerate}
\end{lemma}

\begin{proof}[Proof of Theorem \ref{thm:ept1}]
By Theorem \ref{thm:gen.ewc}, (\ref{thm:ept1::ewc}) implies (\ref{thm:ept1::uc}).  The equivalence of (\ref{thm:ept1::clsd}) and 
(\ref{thm:ept1::opn}) follows by considering complements.\\

(\ref{thm:ept1::uc}) $\Rightarrow$ (\ref{thm:ept1::ewc}): Assume (\ref{thm:ept1::uc}) holds. 
We first show that $\mu$ is computable. By (\ref{thm:ept1::uc}), $\{\mu_n\}_{n \in \N}$ effectively weakly converges
to $\mu$ on uniformly continuous $f \in C(X)$ with $0 \leq f \leq 1$.  Since $\{\mu_n\}_{n \in \N}$ is uniformly computable, it follows 
that condition of Proposition \ref{prop:int} holds. Hence, $\mu$ is computable.

Since $\mu$ is computable, it is possible to compute, on input $N\in\N$, an index of a computable compact $K\subseteq X$ such that $\mu(X\setminus K)<2^{-N}$. Let $\C$ be the class of computable functions in $C(K)$. By the proof of Theorem \ref{thm:gen.ewc}, it suffices to show that $\{\mu_n\}_{n \in \N}$ effectively weakly converges to $\mu$ on $\C$. As every function in $\C$ is a uniformly continuous function in $C_b(X)$, the result follows.

(\ref{thm:ept1::ewc})$\Rightarrow$ (\ref{thm:ept1::opn}): Assume (\ref{thm:ept1::ewc}) holds. 
Then, $\mu$ is computable.   

Let $U \in \Sigma^0_1(X)$.  
We construct a function $g$ as follows.  
Given $r \in \Q$, first wait until $r$ is enumerated into the left Dedekind cut of $\mu(U)$.  
By means of Proposition C.7 of \cite{G05}, we can now compute a non-decreasing sequence $\{t_k\}_{k \in \N}$ of Lipschitz functions so that $0 \leq t_k \leq 1$ and so that $\lim_k t_k = \one_U$.  
 By the Monotone Convergence Theorem, $\lim_k \int_X t_k\ d\mu = \mu(U)$.  Search for $k_0$ so that $\int_X t_{k_0}\ d\mu > r$.  We then compute $N_0, n_0 \in \N$ so that 
$r + 2^{-N_0} < \int_X t_{k_0}\ d\mu$ and 
$|\int_X t_{k_0}\ d\mu - \int_X t_{k_0}\ d\mu_n| < 2^{-N_0}$ when $n \geq n_0$.  Set $g(r) = n_0$.  Thus, when 
$n \geq g(r)$, $\mu_n(U) \geq \int_X t_{k_0}\ d\mu_n > r$.  Therefore, $g$ witnesses that $\liminf_n \mu_n(U)$ is not smaller than $\mu(U)$.

Finally, we note that the construction of $g$ is uniform in that an index of $g$ can be computed from an index of $U$.  Thus, (\ref{thm:ept1::opn}).\\

(\ref{thm:ept1::opn}) $\Rightarrow$ (\ref{thm:ept1::ewc}): Suppose (\ref{thm:ept1::opn}).  
Thus, (\ref{thm:ept1::clsd}).   

Fix $f \in C^c_b(X)$, and suppose $B \in \N$ is an upper bound on $|f|$.  Set 
\[
h = \frac{f + B + 1}{2(B+1)}.
\]
Thus, $0 < h < 1$.  

Let $a_n = \int_X h\ d\mu_n$.  
Let $a = \int_X h\ d\mu$.  
By Lemma \ref{lm:mod.int}, there is a computable witness that 
$\liminf_n a_n$ is not smaller than $a$, and there is a computable witness that 
$\limsup_n a_n$ is not larger than $a$.  Thus, by Proposition \ref{prop:wit.2.mod}, $\lim_n a_n = a$ and 
$\{a_n\}_{n \in \N}$ has a computable modulus of convergence.  It follows that 
$\{\int_{X} f\ d\mu_n\}_{n \in \N}$ has a computable modulus of convergence.

The argument just given is uniform, and so we conclude $\{\mu_n\}_{n \in \N}$ effectively weakly converges to $\mu$.\\ 

(\ref{thm:ept1::opn}) $\Rightarrow$ (\ref{thm:ept1::a.d.}): Suppose (\ref{thm:ept1::opn}).   Thus, (\ref{thm:ept1::clsd}). 

Suppose $A$ is $\mu$-almost decidable.  Let $(U,V)$ be a $\mu$-almost decidable pair for $A$.  
Thus, $\mu(U) = \mu(A) = \mu(X \setminus V)$.  

By (\ref{thm:ept1::opn}), there is a computable witness that $\liminf_n \mu_n(U)$ is not smaller than $\mu(A)$; let $g_1$ be such a witness.  By (\ref{thm:ept1::clsd}), there is also a computable witness that $\limsup_n \mu_n(X \setminus V)$ is not larger than $\mu(A)$; let $g_2$ be such a witness.

Since $\mu_n(U) \leq \mu_n(A) \leq \mu_n(X \setminus V)$, $g_1$ is also a witness that $\liminf_n \mu_n(A)$ is not smaller than $\mu(A)$, and $g_2$ is also a witness that $\limsup_n \mu_n(A)$ is not larger than $\mu(A)$.  
So, by Proposition \ref*{prop:wit.2.mod}, $\{\mu_n(A)\}_{n \in A}$ has a computable modulus of convergence $g$.   

The argument just given is uniform in that an index of $g$ can be computed from an index of $A$.  Hence, (\ref{thm:ept1::a.d.}). \\

(\ref{thm:ept1::a.d.}) $\Rightarrow$ (\ref{thm:ept1::clsd}): Suppose (\ref{thm:ept1::a.d.}).  Thus, $\mu$ is computable.  
Let $C \in \Pi^0_1(X)$.  We construct a function $g$ as follows. 
Let $r\in\Q$. 
Wait until $r$ is enumerated into the right Dedekind cut of $\mu(C)$. 
By Theorem \ref{thm:apprx}, we can then compute an index of a $\mu$-almost decidable 
set $B \supseteq C$ so that $\mu(B) < r$.
Then, compute an $N_0 \in \N$ so that $2^{-N_0} < r- \mu(B)$. 
By (\ref{thm:ept1::a.d.}), we can now compute an $n_0 \in \N$ so that $|\mu_n(B)-\mu(B)|<2^{-N_0}$ when $n\geq n_0$. 
Set $g(r)=n_0$.
Thus, when $n\geq g(r)$, $\mu_n(B)<r$ and so $\mu_n(C) \leq \mu(B)<r$. 
It follows that $g$ is a witness that $\limsup_n\mu_n(C)$ is not larger than $\mu(C)$.
\end{proof}

\subsection{Effective Convergence in the Prokhorov Metric}

We say that a sequence $\{\mu_n\}_{n\in\N}$ in $\M(X)$ \emph{converges effectively in the Prokhorov metric} $d_P$ to a measure $\mu$ if there is a computable function $\epsilon:\N\rightarrow\N$ such that for every $n,N\in\N$, $n\geq\epsilon(N)$ implies $d_P(\mu_n,\mu)<2^{-N}$. 
Since $\M(X)$ forms a computable metric space under $d_P$, it follows that every uniformly computable sequence of measures in $\M(X)$ converges to a computable measure in $d_P$. 

We now proceed to show that Theorem 4.1 in \cite{R24} generalizes to computable Polish spaces.

\begin{theorem}\label{thm:cpm}
    Suppose $\{\mu_n\}_{n\in\N}$ is uniformly computable. The following are equivalent:
    \begin{enumerate}
        \item $\{\mu_n\}_{n\in\N}$ is effectively weakly convergent;
        \item $\{\mu_n\}_{n\in\N}$ converges effectively in $d_P$.
    \end{enumerate}
\end{theorem}

First, we need the following lemma.

\begin{lemma}\label{lem:eoc}
    Let $\mu\in\M(X)$ be computable, and let $s>0$ be rational. It is possible to compute an open cover of $X$ consisting of open balls with radius less that $s$, each of which is a $\mu$-almost decidable set.
\end{lemma}

\begin{proof}
    Adapt the proof of Lemma 5.1.1 in \cite{HR09b} by replacing $\R^+$ with $(0,s)$.
\end{proof}

The proof of the classical version of Theorem \ref{thm:cpm} makes use of the classical Portmanteau Theorem as well as a classical version of Lemma \ref{lem:eoc}. 
As we shall see below, Theorem \ref{thm:cpm} makes use of the effective Portmanteau Theorem as well as Lemma \ref{lem:eoc}. 
However, before proving Theorem \ref{thm:cpm}, we also require the following lemma. 

\begin{lemma}\label{lem:cn}
    If $C\in\Pi^0_1(X)$, then $\ol{B(C,s)}\in\Pi^0_1(X)$ for any rational $s>0$.
\end{lemma}

\begin{proof}
    Let $\{I_i\}_{i\in\N}$ be an enumeration of all rational open balls of $X$. 
    Then, for each $i\in\N$, $I_i=B(a_i,r_i)$ for some $a_i,r_i\in\Q$ with $r_i>0$. 
    Let $A=\ol{B(C,s)}$, and let $E_A=\{i\in\N:I_i\cap A=\emptyset\}$. 
    It suffices to show that $E_A$ is c.e. \ 
    Now, for each $i\in\N$, we enumerate $i$ into $E_A$ whenever $B(a_i,r_i)\cap A=\emptyset$. 
    Note that $B(a_i,r_i)\cap A=\emptyset$ if and only if $d(a_i,A)>r_i$, which occurs if and only if $d(a_i,C)>r_i+s$. 
    Since $C\in\Pi^0_1(X)$, $x\mapsto d(x,C)$ is lower semicomputable (Theorem 5.1.2 in \cite{Weihrauch.2000}). 
    Thus, the enumeration is effective. 
    Since $s>0$ was arbitrary, the result follows.
\end{proof}

\begin{proof}[Proof of Theorem \ref{thm:cpm}]
    Suppose that $\{\mu_n\}_{n\in\N}$ converges effectively in $d_P$ to $\mu$. Then, $\mu$ is computable, and we have a computable function $\epsilon:\N\rightarrow\N$ such that for all $N\in\N$ and all $n\geq\epsilon(N)$, $d_P(\mu_n,\mu)<2^{-N}$. In particular, for any $C\in\Pi^0_1(X)$ and all $n\geq\epsilon(N)$,
    \[
        \mu_n(C)\leq\mu(B(C,2^{-N}))+2^{-N}\text{ and }\mu(C)\leq\mu_n(B(C,2^{-N}))+2^{-N}.
    \]
    By Theorem \ref{thm:ept1}, it suffices to compute an index of a witness that $\limsup_n\mu_n(C)$ is not larger than $\mu(C)$ from an index $e\in\N$ of $C\in\Pi^0_1(X)$.

    Fix $r\in\Q$. 
    Wait until $r$ is enumerated into the right Dedekind cut of $\mu(C)$.
    Once $r$ has been enumerated into the right Dedekind cut of $\mu(C)$, search for the first $M_0\in\N$ so that $r-\mu(C)>2^{-M_0}$. 
    By Lemma \ref{lem:cn} and the fact that $\mu$ is computable, $\mu(\ol{B(C,2^{-N})})$ is right-c.e. for any $N\in\N$. 
    Thus, search for the first $N_0$ such that $r-\mu(\ol{B(C,2^{-N_0})})>2^{-M_0}$. 
    Let $M=M_0+N_0+1$, and let $n_0=\epsilon(M)$. 
    Therefore, for all $n\geq n_0$,
    \[
        \mu_n(C)\leq\mu(B(C,2^{-M}))+2^{-M}\leq\mu(\ol{B(C,2^{-M})})+2^{-M}<r-2^{-M}+2^{-M}=r.
    \]
    It follows that $n_0$ is an index of a witness that $\limsup_n\mu_n(C)$ is not larger than $\mu(C)$.

    Next, suppose that $\{\mu_n\}_{n\in\N}$ effectively weakly converges to $\mu$. Then, $\mu$ is computable. By Theorem \ref{thm:ept1}, we can compute for every $\mu$-almost decidable $A$ an index of a modulus of convergence of $\{\mu_n(A)\}_{n\in\N}$ from a $\mu$-almost decidable index of $A$.

    We build the function $\epsilon:\N\rightarrow\N$ by the following effective procedure. 
    First, let $N\in\N$. 
    By Lemma \ref{lem:eoc}, we can compute a sequence $\{B_j\}_{j\in\N}$ of uniformly $\mu$-almost decidable rational open balls in $X$ with radius less than $2^{-(N+3)}$ such that $\bigcup_{j=1}^{\infty}B_j=X$. 
    Search for the first $k_0$ so that $\mu(\bigcup_{j=1}^{k_0}B_j)\geq\mu(X)-2^{-(N+2)}$. 
    Let
    \[
        \A=\left\{\bigcup_{j\in J}B_j:J\subseteq\{1,\ldots,k_0\}\right\}.
    \]
    Then, $\A$ is a finite collection of $\mu$-almost decidable sets. 
    Define $\epsilon(N)$ to be the smallest index so that $|\mu_n(A)-\mu(A)|<2^{-(N+2)}$ for every $A\in\A$ and every $n\geq\epsilon(N)$. 
    Thus,
    \[
        \mu_n\left(X\setminus\left(\bigcup_{j=1}^{k_0}B_j\right)\right)\leq\mu\left(X\setminus\left(\bigcup_{j=1}^{k_0}B_j\right)\right)+2^{-(N+2)}\leq2^{-(N+1)}
    \]
    for all $n\geq\epsilon(N)$.

    To see that $\epsilon$ is the desired function, fix $E\in\B(X)$. 
    Let
    \[
        A_0=\bigcup\{B_j:(j\in\{1,\ldots,k_0\})\wedge(B_j\cap E\neq\emptyset)\}.
    \]
    Then, $A_0\in\A$ and satisfies the following properties:
    \begin{enumerate}
        \item $A_0\subset B(E,2^{-(N+2)})\subseteq B(E,2^{-N})$.
        \item $E\subset A_0\cup\left(X\setminus\left(\bigcup_{j=1}^{k_0}B_j\right)\right)$.
        \item $|\mu_n(A_0)-\mu(A_0)|<2^{-(N+2)}$ for all $n\geq\epsilon(N)$.
    \end{enumerate}
    Therefore, for all $n\geq\epsilon(N)$,
    \begin{align*}
        \mu(E)&\leq\mu(A_0)+2^{-(N+2)}\\
        &<\mu_n(A_0)+2^{-(N+1)}\\
        &\leq\mu_n(B(E,2^{-(N+2)}))+2^{-(N+1)}<\mu_n(B(E,2^{-N}))+2^{-N}
    \end{align*}
    and
    \[
        \mu_n(E)\leq\mu_n(A_0)+2^{-(N+1)}<\mu(A_0)+2^{-N}\leq\mu(B(E,2^{-N}))+2^{-N}.
    \]
    Since $E\in\B(X)$ was arbitrary, it follows that $d_P(\mu_n,\mu)<2^{-N}$ for all $n\geq\epsilon(N)$, as desired.
\end{proof}

\section{An Effective Version of Prokhorov's Theorem}

In the previous two sections, we developed the framework necessary to formulate an effective version of Prokhorov's Theorem. 

Before we state the main result of this paper, note that for any computable measure $\nu\in\M(X)$, $\nu(B)$ is left-c.e. (right-c.e.) uniformly in an index of $B$ if $B$ is a rational open (closed) ball. However, it is not guaranteed that all rational open and closed balls have computable $\nu$-measure. The situation becomes more complicated when looking at rational open and closed balls across a computable sequence of measures. Another concern is the fact that a rational closed ball need not to be the closure of a rational open ball. Thankfully, the following lemma provides an effective solution to this problem.

\begin{lemma}\label{lem:a.d.basis}
    Let $\{\mu_n\}_{n\in\N}$ be a computable sequence of probability measures in $\M(X)$. 
    There is a computable sequence of radii $\{r_j\}_{j\in\N}$ such that $\{B(s_i,r_j)\}_{i,j\in\N}$ is a basis for the topology on $X$ with the following properties:
    \begin{enumerate}
        \item $B(s_i,r_j)$ is $\mu_n$-almost decidable uniformly in $i$, $j$, and $n$.
        \item $\overline{B(s_i,r_j)}=\overline{B}(s_i,r_j)$ for each $i$ and $j$.
    \end{enumerate}
\end{lemma}

\begin{proof}
    For each $n$, $i$, and $k$, $U_{n,i,k}=\{r\in(0,\infty):\mu_n(\overline{B}(s_i,r))<\mu_n(B(s_i,r))+1/k\}$ is a dense c.e. open subset of $(0,\infty)$ uniformly in $n$, $i$, and $k$ for the following reasons:
    \begin{enumerate}
        \item $\partial B(s_i,r)\cap\partial B(s_i,r')=\emptyset$ whenever $r\neq r'$;
        \item $\{r\in(0,\infty):\mu_n(\partial B(s_i,r))\geq1/k\}$ is finite for each $n$ since each $\mu_n$ is finite;
        \item $\mu_n$ is computable uniformly in $n$.
    \end{enumerate}

    Now, for each $i$ and $j$, $V_{i,j}=(0,\infty)\setminus\{d(s_i,s_j)\}$ is a dense $\Sigma^0_1$ subset of $(0,\infty)$ uniformly in $i$ and $j$. Moreover, for each $i$, $W_i=\{r\in(0,\infty):\overline{B(s_i,r)}=\overline{B}(s_i,r)\}$ is a dense $\Sigma^0_1$ subset of $(0,\infty)$ uniformly in $i$. It follows that
    \[
    R=\bigcap_{n,i,k}U_{n,i,k}\cap\bigcap_{i,j}V_{i,j}\cap\bigcap_iW_i
    \]
    is a dense $\Pi^0_2$ subset of $(0,\infty)$. The result follows from the computable Baire Category Theorem.
\end{proof}

We will call the basis $\{B(s_i,r_j)\}_{i,j\in\N}$ obtained from Lemma \ref{lem:a.d.basis} a $\mu_n$\emph{-regular ball basis} for $X$. With this, we can state the main theorem.

\begin{theorem}[Effective Prokhorov's Theorem]\label{thm:etp}
    Let $\{\mu_n\}_{n\in\N}$ be a computable sequence of probability measures in $\M(X)$. If $\{\mu_n\}_{n\in\N}$ is effectively tight, then $\{\mu_n\}_{n\in\N}$ contains an effectively weakly convergent subsequence. Furthermore, such a sequence can be computed from an effective tightness witness of $\{\mu_n\}_{n\in\N}$ and a $\mu_n$-regular ball basis for $X$.
\end{theorem}

\begin{corollary}\label{cor:etp}
    Let $\{\mu_n\}_{n\in\N}$ be a computable sequence of probability measures in $\M(X)$. If $\{\mu_n\}_{n\in\N}$ is effectively tight, then $\{\mu_n\}_{n\in\N}$ contains at least one computable cluster point.
\end{corollary}

The theorem above indicates that effectively tightness is sufficient for the existence of a computable cluster point. However, as we will see, we also need a $\mu_n$-regular ball basis to compute an effectively weakly convergent subsequence. This is consistent with the fact that Prokhorov's Theorem is the Bolzano-Weierstrass principle on $\M(X)$ since there is no uniform procedure to compute an effectively weakly convergent subsequence of a tight computable sequence in $\M(X)$. In fact, we can show the following.

\begin{proposition}
    There exists a computable sequence $\{\mu_n\}_{n\in\N}$ of probability measures on $X$ that is tight and contains an effectively weakly convergent subsequence but is not effectively tight.
\end{proposition}

\begin{proof}
    Let $\{q_n\}_{n\in\N}$ be a computable increasing sequence of positive rationals such that $x=\lim_nq_n$ is not computable. Consider the sequence
    \[
    \mu_n=\begin{cases}
        \delta_0 & \text{if} \ n \ \text{is even}\\
        \delta_{q_{(n-1)/2}} & \text{if} \ n \ \text{is odd}\\
    \end{cases}
    \]
    on $\M(\R)$. Clearly, the sequence $\{\mu_{2n}\}_{n\in\N}$ is an effectively weakly convergent subsequence of $\{\mu_n\}_{n\in\N}$. Furthermore, $\{\mu_n\}_{n\in\N}$ is tight because $\supp\{\mu_n\}\subseteq[-x,x]$ for all $n\in\N$.
    
    By way of contradiction, suppose $\{\mu_n\}_{n\in\N}$ is effectively tight. By Proposition \ref{prop:tight}, $\{\mu_n\}_{n\in\N}$ admits an effective tail-tightness witness $\{(K_N,n_N)\}_{N\in\N}$ where each $K_N$ consisting of finite unions of rational closed balls in $\R$. Since an effective tail-tightness witness may ignore an initial segment of $\{\mu_n\}_{n\in\N}$, we may assume without loss of generality that $n_N=2N$. Note that for each $k\geq N$, $\mu_{2k+1}(K_N)>1-2^{-N}$ if and only if $q_k\in K_N$. Thus, we may assume without loss of generality that, for each $N\in\N$, $K_N=B_0\cup C_N$, where $B_0$ is a rational closed ball containing $0$ and $q_k\in C_N$ for all $k\geq N$. Let $C_N'=\bigcap_{i\leq N}C_i$. Then, $C_N'\supseteq C_{N+1}'$ for all $N\in\N$. Furthermore, since $\lim_kq_k=x$, we may assume without loss of generality that $\lim_N\diam(C_N')=0$. It follows that $\bigcap_NC_N'=\{x\}$.

    Let $u_N=\max C_N'+2^{-N}$ for each $N\in\N$. Since $C_N$ is a finite union of rational closed balls for each $N\in\N$, $u_N$ is rational for each $N\in\N$. Since $C_N'$ is nested, it follows that $\{u_N\}_{n\in\N}$ is a computable decreasing sequence of rationals. Since $\bigcap_NC_N'=\{x\}$, it follows that $\lim_Nu_N=x$. Therefore, $x$ is computable: a contradiction. As a result, $\{\mu_n\}_{n\in\N}$ is not effectively tight.
\end{proof}

\begin{proof}[Proof of Theorem \ref{thm:etp}]
    Suppose $\{\mu_n\}_{n\in\N}$ is effectively tight. By assumption, it is possible to compute an effective tightness witness $\{K_N\}_{N\in\N}$ of $\{\mu_n\}_{n\in\N}$ consisting of finite unions of rational closed balls. Without loss of generality, we may assume $K_N\subseteq K_{N+1}$ for all $N\in\N$. 

    Let $\A=\{B(x_i,r_j)\}_{i,j\in\N}$ be a $\mu_n$-regular basis for $X$. We now define the function $\alpha:\N\rightarrow\N$ be given recursively as follows: Let $\alpha(0)=n_0=0$. Assuming $\alpha(s)$ has been defined at stage $s\in\N$, search for an index $j_{s+1}$ such that $r_{j_{s+1}}<2^{-(s+1)}$. Compute a minimal cover $\U_{s+1}=\{U_{s+1,1},\ldots,U_{s+1,M_{s+1}}\}$ of $K_{s+1}$ consisting of open balls in $\A$ of radius less than $r_{j_{s+1}}$. Then, search for the first index $n_{s+1}>\alpha(s)$ such that for all $B\in\U_{s+1}$,
    \[
    |\mu_{n_{s+1}}(B)-\mu_{\alpha(s)}(B)|<2^{-(s+1)}.
    \]
    Let $\alpha(s+1)=n_{s+1}$.

    By construction, $\alpha$ is stricly increasing. Since $\{r_j\}_{j\in\N}$ is a computable sequence, the search for $j_s$ terminates at each stage $s$. Since each $K_s$ is computably compact, we can effectively obtain $\U_s$ at each stage $s$. Furthermore, since there are only finite many $B\in\U_s$ and each $B\in\U_s$ is uniformly $\mu_n$-almost decidable, the search for $n_s$ terminates at each stage $s$. Therefore, $\alpha$ is a computable increasing function. Hence, $\{\mu_{\alpha(n)}\}_{n\in\N}$ is a computable subsequence of $\{\mu_n\}_{n\in\N}$.

    We have left to show that $\{\mu_{\alpha(n)}\}_{n\in\N}$ effectively weakly converges. 
    Let $\mathcal{R}$ be the ring consisting of finite unions and relative complements of basis elements in $\A$. Note that $\sigma(\A)=\sigma(\mathcal{R})=\B(X)$. 
    Define the map $\eta:\mathcal{R}\rightarrow[0,1]$ given by $\eta(B)=\lim_s\mu_{\alpha(s)}(B)$ for each $B\in\mathcal{R}$.

    \begin{claim}\label{cl:1}
    $\eta(B)$ is a computable real uniformly in an index of $B\in\mathcal{R}$.
    \end{claim}
    \begin{proof}[Proof of Claim]
    Fix $k\in\N$. Let $B\in\A$. Search for the first $s\geq k+2$ such that $B$ is covered by $\U_{s}=\{U_{s,1},\ldots,U_{s,M_s}\}$. Note that $B\cap K_s\subseteq K_s\subseteq\bigcup_{\ell\leq M_s}U_{s,\ell}$. Let $s_0=\max\{s,k+\lceil \log_2(M_s)\rceil+1\}$. Then, for any $m\geq n\geq s_0$,
    \begin{align*}
        |\mu_{\alpha(n)}(B)-\mu_{\alpha(m)}(B)|&\leq|\mu_{\alpha(n)}(B\cap K_s)-\mu_{\alpha(m)}(B\cap K_s)|\\
        &\hspace{0.5in}+\mu_{\alpha(n)}(X\setminus K_s)+\mu_{\alpha(m)}(X\setminus K_s)\\
        &<\sum_{\ell=1}^{M_s}|\mu_{\alpha(n)}(U_{s,\ell})-\mu_{\alpha(m)}(U_{s,\ell})|+2\cdot2^{-s}\\
        &\leq\sum_{i=n}^{m-1}\sum_{\ell=1}^{M_s}|\mu_{\alpha(i)}(U_{s,\ell})-\mu_{\alpha(i+1)}(U_{s,\ell})|+2^{-(k+1)}\\
        &<\sum_{i=n}^{m-1}M_s2^{-(i+1)}+2^{-(k+1)}\\
        &<M_s2^{-n}+2^{-(k+1)}\\
        &\leq2^{-(k+1)}+2^{-(k+1)}\\
        &=2^{-k}
    \end{align*}
    
    Now, let $B_1,B_2\in\A$. Search for the first $s\geq k+4$ such that $B_1\cup B_2$ is covered by $\U_s$. Let $s_0=\max\{s,k+\lceil \log_2(M_s)\rceil+3\}$. Consider the following cases.
    
    \begin{enumerate}
        \item If $B=B_1\cup B_2$, then $|\mu_{\alpha(n)}(B)-\mu_{\alpha(m)}(B)|<2^{-k}$
        for all $m\geq n\geq s$.
        \item If $B=B_1\cap B_2$, then
        \begin{align*}\hspace{0.2in}
        |\mu_{\alpha(n)}(B)-\mu_{\alpha(m)}(B)|&\leq|\mu_{\alpha(n)}(B_1)-\mu_{\alpha(m)}(B_1)|\\
        &\hspace{0.3in}+|\mu_{\alpha(n)}(B_2)-\mu_{\alpha(m)}(B_2)|\\
        &\hspace{0.6in}+|\mu_{\alpha(n)}(B_1\cup B_2)-\mu_{\alpha(m)}(B_1\cup B_2)|\\
        &<3\cdot2^{-(k+2)}<2^{-k}
        \end{align*}
        for all $m\geq n\geq s$.
        \item If $B=B_1\setminus B_2$, then
        \begin{align*}\hspace{0.2in}
        |\mu_{\alpha(n)}(B)-\mu_{\alpha(m)}(B)|&\leq|\mu_{\alpha(n)}(B_1)-\mu_{\alpha(m)}(B_1)|\\
        &\hspace{0.5in}+|\mu_{\alpha(n)}(B_1\cap B_2)-\mu_{\alpha(m)}(B_1\cap B_2)|\\
        &<2^{-(k+2)}+3\cdot2^{-(k+2)}=2^{-k}
        \end{align*}
        for all $m\geq n\geq s$.
    \end{enumerate}
    In all possible cases, $\{\mu_{\alpha(n)}(B)\}_{n\in\N}$ is effectively Cauchy uniformly in an index of $B\in\mathcal{R}$. Therefore, the claim holds.
\end{proof}

By Claim \ref{cl:1}, $\eta$ is a premeasure on $\mathcal{R}$. Since $\sigma(\A)=\sigma(\mathcal{R})=\B(X)$, $\eta$ is a premeasure on a ring generating $\B(X)$. Thus, $\eta$ extends to a Borel measure $\mu$ on $X$ by the Carath\'{e}odory Extension Theorem. Furthermore, since each $\mu_{\alpha(n)}$ is a probability measure, $\mu$ is the unique probability measure that extends $\eta$ to $\B(X)$.

We need to show that $\mu$ is computable. Given $U\in\Sigma^0_1(X)$, $U=\bigcup_{\left<i,j\right>\in A}B(x_i,r_j)$ for some c.e. set $A\subseteq\N$. Since $\mu(B)$ is computable uniformly in an index of $B\in\A$ and by continuity of measure, it is possible to compute an increasing sequence of rationals converging to $\mu(U)$. It follows that $\mu(U)$ is left-c.e. uniformly in an index of $U$. As a result, $\mu$ is a computable probability measure. 

We have left to show that, from an index of $U\in\Sigma^0_1$, it is possible to compute an index of a witness that $\liminf_s\mu_{\alpha(s)}(U)$ is not smaller than $\mu(U)$. To see this, let $r\in\Q$ and $U\in\Sigma^0_1(X)$. Wait until $r$ is enumerated into the left Dedekind cut of $\mu(U)$. Once done, search for the first $N_0\in\N$ and the first index of a set $B_0\in\mathcal{R}$ such that $B_0\subseteq U$ and $r+2^{-N_0}<\mu(B_0)$. Next, search for the first $s_0\in\N$ such that $|\mu(B_0)-\mu_{n_{\alpha(s)}}(B_0)|<2^{-N_0}$ for all $s\geq s_0$. Set $g(r)=s_0$. then, for all $s\geq g(r)$,
    \[
        \mu_{\alpha(s)}(U)\geq\mu_{\alpha(s)}(B_0)\geq\mu(B_0)-|\mu(B_0)-\mu_{\alpha(s)}(B_0)|>r.
    \]
    Therefore, $\{\mu_{\alpha(s)}\}_{s\in\N}$ effectively weakly converges to $\mu$ by Theorem \ref{thm:ept1}.
\end{proof}

\section{Conclusion}
Earlier, we provided a framework to study the effective theory of weak convergence and tightness of measures on computable Polish spaces. The completeness assumption is crucial, as it guarantees that individual measures are effectively tight. We conjecture that there exists a computable measure on an incomplete computable metric space that is not effectively tight.

In Section 4, we observed that effective tightness was required to show the equivalence of effective weak convergence and uniform effective weak convergence in $\M(X)$ for a computable Polish space $X$. However, without effective tightness, it is possible that the equivalence may not hold. This leads us to the following question:

\begin{question}
    Is there an incomplete computable metric space in which it is possible to construct a computable sequence of measures that effectively weakly converges but does not so uniformly?
\end{question}

We also proved an effective version of Prokhorov's Theorem in Section 5. From this result, we can conclude that the effective tightness condition for computable sequences of probability measures on computable Polish spaces is quite difficult to achieve. Not only that, we were able to create a separation between effective tightness and the existence of an effectively weakly convergent subsequence. Hence, we are able to see a separation in the classical result of Polish spaces, where tightness is both a necessary \emph{and} sufficient condition for relative sequential compactness. 

Earlier, we pointed out that Prokhorov's Theorem is the Bolzano-Weierstrass principle on $\M(X)$. This raises the following questions in the context of Weihrauch reducibility.
\begin{question}
    Given a computable Polish space $X$, is $\mathsf{BWT}_{\M(X)}\equiv_W\mathsf{BWT}_{X}$? Is $\mathsf{BWT}_{\M(X)}\equiv_{sW}\mathsf{BWT}_{X}$?
\end{question}
It is natural, then, to ask the following question in the reverse mathematics setting.
\begin{question}
    Let $\mathsf{PT}_1$ be the statement ``Every tight sequence of probability measures on a Polish space $X$ contains a weakly convergent subsequence,'' and let $\mathsf{PT}_2$ be the statement ``Every sequence of probability measures on $X$ containing a weakly convergent subsequence is tight.'' What is the weakest subsystem of second-order arithmetic required to prove $\mathsf{PT}_1\leftrightarrow\mathsf{PT}_2$?
\end{question}

To approach a solution to this question, we need to consider both implications in Prokhorov's Theorem separately. This is because relatively weakly sequentially compact families of probability measures are guaranteed to be tight only when considering measures on a Polish space.

\bibliographystyle{plain}
\bibliography{ourbib}

\end{document}